\documentclass{amsart}

\usepackage{packages, formatting, notation}

\title{Serre Functors and Local Duality for Affine Quotients}
\author{Ivan Noden}
\date{}

\begin{document}
  \maketitle

  \setcounter{tocdepth}{1}
  \tableofcontents

  \sec{Introduction}

  \ssec{Overview}
  The purpose of this short note is to study Serre functors
  of categories of the form \(\QCoh(\Y)\)
  where \(\Y = \Spec A/G\) is the quotient of
  a reductive group \(G\) acting on \(\Spec A\) with a
  unique closed orbit.

  \sssec{}

  Recall, if \(\C\) is a compactly generated DG category then there exists
  a continuous functor \(\Se_\C: \C \to \C\) defined by the
  property that
  \begin{displaymath}
    \dual{\Hom_{\C}(x,y)} \simeq \Hom_{\C}(y, \Se_\C(x))
  \end{displaymath}
  for all \(x, y \in \C\) compact. The functor \(\Se_\C\)
  is known as the Serre functor of \(\C\).

  Such functors have been studied since the classical
  theorems of Serre and Grothendieck and their
  generalisation to triangulated categories is seen
  in works such as that of Bondal--Kapranov \cite{bondal_kapranov:1990}.
  Traditionally, in these instances, it is required \(\Se_\C\)
  be an equivalence but this assumption is often
  dropped in the DG setting where such functors have been
  studied and used to great effect by, for instance,
  Beraldo in \cite{beraldo:2021} and Gaitsgory--Yom Din in
  \cite{gaitsgory_yom_din:2017}.

  \sssec{}

  When \(\C = \QCoh(\Y)\) for \(\Y\) as above,
  the Serre functor is given by tensoring
  with an object \(\S_\Y \in \QCoh(\Y)\).
  We find that \(\S_\Y\)
  is the local cohomology
  of \(\omega_\Y\) at the unique closed
  orbit. Using this, we are able to
  develop analogues of Matlis and local
  duality.

  In the remainder of this introduction,
  we recall the classical theorems
  of local and Matlis duality 
  in order to highlight parallels with our results on affine quotients.
  After this, we briefly
  study the specific example of \(\Y \simeq \A^1/\G_m\)
  to illustrate these connections before stating our main results.

  \sssec{}

  Suppose \((A, \fm, k)\) is a local
  ring and \(I\) is the injective hull of \(k\) in the
  abelian category of \(A\)--modules. Then Matlis proves that,
  for an \(A\)--module \(M\),
  if \(M\) is finitely generated, \(\usHom_A(M, I)\)
  is Artinian. Conversely, if \(M\) is Artinian, then \(\usHom_A(M, I)\)
  is finitely generated. Furthermore, if \(M\) is
  either finitely generated or Artinian, \(M \simeq \usHom_A(\usHom_A(M,I), I)\).

  One notes that, if \(M\) is finitely generated then
  its local cohomology modules \(H^i_\fm M\) are Artinian.
  Thus, if we let \(\Gamma_\fm\) denote the derived
  local cohomology functor, the functor
  \(\uHom_A(-, I)\) defines an
  involutive equivalence between
  \(\Coh(\Spec A)\) and \(\Gamma_\fm(\Coh(\Spec A))\).

  The local duality theorem of Grothendieck
  then tells us that the functor \(\uHom_A(\Gamma_\fm(-), I)\)
  is representable by the normalised dualising complex \(\omega_A\)
  of \(A\).

  The connection with Serre functors is that, when
  the category \(\QCoh(\Spec A)\) is \emph{reflexive}
  (a condition which ensures the Serre functor is nicely behaved),
  then the Serre functor on \(\QCoh(\Spec A)\) is given
  by tensoring with \(I\).

  \sssec{}

  Let us now consider a baby example of the affine quotient case.
  We take \(\Y = \A^1/\G_m\) with the usual dilation action
  of \(\G_m\) on \(\A^1\). Put \(\pi: \Y \to B\G_m\)
  as the natural map and note that \(\QCoh(\Y)\) is
  compactly generated by \(\pi^\ast k_n\) where \(k_n\)
  is the representation of \(\G_m\) with highest weight
  \(n \in \Z\). As the compact objects of \(\QCoh(\Y)\)
  are dualisable, the Serre functor is given by tensoring
  with \(\S \in \QCoh(\Y)\) characterised by the property
  that
  \begin{displaymath}
    \Hom_{\QCoh(\Y)}(\pi^\ast k_n, \S) \simeq \dual{\Hom_{\QCoh(\Y)}(\O_\Y, \pi^\ast k_n)} \simeq \dual{\Hom_{\Rep(G)}(k_0, \pi_\ast \pi^\ast k_n)}
    \simeq
    \begin{cases}
      \dual{(k \cdot x^n)}, & \text{if } n \geq 0, \\
      0, & \text{otherwise.}
    \end{cases}
  \end{displaymath}
  
  One can then observe that
  \begin{displaymath}
    \S \simeq \cone(\pi^\ast k_1 \to \colim_n \pi^\ast k_{-n}).
  \end{displaymath}
  This object has a geometric interpretation.
  Let \(j: \punct{\Y} \hto \Y\) be the obvious 
  open immersion corresponding to \(\A^1 - \{0\}\)
  and define \(\QCoh(\Y)_0\) to be the full
  subcategory given by the kernel of \(j^\ast\).
  The embedding \(\QCoh(\Y)_0 \hto \QCoh(\Y)\)
  admits a right adjoint:
  \begin{displaymath}
    \Gamma_0 = \ker(\id_{\QCoh(\Y)} \to j_\ast j^\ast).
  \end{displaymath}
  We remark that \(\Gamma_0 \F \simeq \Gamma_0 \O_\Y \otimes \F\)
  by the projection formula.
  Observe then that \(\S \simeq \Gamma_0(\pi^\ast k_1 [1])\).
  Interestingly, \(\pi^\ast k_1 \simeq \Psi(\omega_\Y)\)
  as \(\Omega_{\A^1} \simeq \O_{\A^1} \cdot \mathrm{d}x\)
  and \(x\) has weight \(1\).
  Thus,
  \begin{displaymath}
    \S \simeq \Gamma_0 \Psi(\omega_\Y)[1].
  \end{displaymath}

  \sssec{}

  By its defining property \(\uHom_{\QCoh(\Y)}(\S, \S) \simeq \O_\Y\)
  and since, \(\Coh(\Y) \simeq \Perf(\Y)\), for all \(\F \in \Coh(\Y)\),
  \(\uHom_{\QCoh(\Y)}(\uHom_{\QCoh(\Y)}(\F, \S), \S) \simeq \F\).
  Furthermore, for \(\F \in \Coh(\Y)\),
  we have
  \begin{displaymath}
    \uHom_{\QCoh(\Y)}(\F, \S) \simeq \dual{\F} \otimes \S \simeq \Gamma_0(\Psi(\omega_\Y) \otimes \dual{\F})[1]
  \end{displaymath}
  and
  \begin{displaymath}
    \uHom_{\QCoh(\Y)}(\Gamma_0\F, \S) \simeq \uHom_{\QCoh(\Y)}(\F \otimes \dual{\Psi(\omega_\Y)} \otimes \S[-1], \S) \simeq \dual{\F} \otimes \Psi(\omega_\Y)[1].
  \end{displaymath}
  Thus, \(\uHom_{\QCoh(\Y)}(-, \S_\Y)\)
  defines an involutive equivalence between
  \(\Gamma_0(\Coh(\Y))\) and \(\Coh(\Y)\)

  \ssec{Main results}

  While the example of \(\A^1/\G_m\) relied on the smoothness of \(\A^1\),
  this assumption is not necessary. What we really use is
  the existence of a unique closed orbit under the \(\G_m\)
  action. Let us spell out the more general case.

  \sssec{}
  Let \(G\) be a reductive group acting on the classical affine scheme
  \(\Spec A\) of finite type. Consider the stack \(\Y = \Spec A/G\).
  The category \(\QCoh(\Y)\) is proper (that is,
  \(\Hom_{\QCoh(\Y)}(\F_1, \F_2)\) has finite dimensional
  total cohomology whenever \(\F_1\) and \(\F_2\) are
  compact) if \(\Spec A\) has finitely many closed orbits
  under the action of \(G\). We focus specifically
  on the case where there is a unique closed orbit,
  in which case \(\Y \simeq \Spec B/H\) where
  \(\Spec B\) has a unique closed orbit under the
  action of reductive \(H\) and, furthermore,
  this closed orbit is a fixed point.
  Therefore, we will assume \(\Spec A\) has
  a unique closed orbit which is a fixed point.

  Not only is \(\QCoh(\Y)\) proper but its compact
  objects are dualisable so the Serre functor
  is given by tensoring with some \(\S_\Y \in \QCoh(\Y)\).

  \sssec{}

  Let \(i: BG \hto \Y\) be the closed immersion of the
  fixed point and \(j: \punct{\Y} \hto \Y\) its open complement.
  The full subcategory of \(\QCoh(\Y)\) given by the kernel
  of \(j^\ast\) is denoted \(\QCoh(\Y)_{BG}\). It is
  the subcategory of quasicoherent sheaves
  set-theoretically supported at the fixed point.

  The embedding \(\QCoh(\Y)_{BG} \hto \QCoh(\Y)\)
  admits a right adjoint, denoted \(\localcoh\)
  and called the \emph{local cohomology functor}.
  Explicitly,
  \begin{displaymath}
    \localcoh \F \simeq \ker(\F \to j_\ast j^\ast \F).
  \end{displaymath}
  By the projection formula, \(\localcoh \F \simeq \localcoh \O_\Y \otimes \F\).

  \sssec{}

  Our main result is an explicit description of \(\S_\Y\):

  \begin{mainthm}\label{thm:serre_functor}
    We have \(\S_\Y \simeq \localcoh \Psi(\omega_\Y)[\dim G]\).
  \end{mainthm}

  The main idea behind the proof is showing that \(\S_\Y\)
  is supported entirely at the fixed point. From this,
  its defining property allows us to characterise it
  as above.

  We remark that this description for \(\S_\Y\) was
  found by Beraldo in \cite{beraldo:2021} in the
  specific case where \(\Spec A\) is the
  nilpotent cone in the Lie algebra of \(G\).

  \sssec{}

  Using this description, we are able to prove an analogue
  of Matlis and local duality. Writing
  \begin{displaymath}
    \Dmat = \uHom_{\QCoh(\Y)}(-, \S_\Y) \text{ and } \Dse = \uHom_{\QCoh(\Y)}(-, \Psi(\omega_\Y))
  \end{displaymath}
  we prove:

  \begin{mainthm}\label{thm:matlis_local_duality}
    For all \(\F\) in \(\Coh(\Y)\),
    \begin{displaymath}
      \Dmat(\localcoh \F) \simeq \Dse(\F)[\dim G] \text{ and } \Dmat(\F) \simeq \localcoh(\Dse(\F)) [\dim G],
    \end{displaymath}
    hence \(\Dmat\) defines involutive equivalences 
    \(\op{\Coh(\Y)} \to \localcoh(\Coh(\Y))\) and \(\op{\localcoh(\Coh(\Y))} \to \Coh(\Y)\).
  \end{mainthm}

  The proof of this of course uses \Cref{thm:serre_functor}.
  The main step aside from this is showing that, for
  \(\F \in \Coh(\Y)\), \(\Dse(j_\ast j^\ast \F) \simeq 0\)
  from which we can deduce that \(\Dmat(\localcoh \F) \simeq \Dse(\F)[\dim G]\).
  To do this, we show first that \(\Dmat(\localcoh \F)\)
  lands in \(\Coh(\Y)\) which forces \(\Dse(j_\ast j^\ast \F) \in \Coh(\Y)\).
  We can then exploit \Cref{prop:k_generates}
  which tells us that \(\uHom_{\QCoh(\Y)}(i_\ast \O_{BG}, -)\)
  is conservative on \(\Coh(\Y)\).

  \sssec{}

  We use both \Cref{thm:serre_functor} and
  \Cref{thm:matlis_local_duality} to also calculate
  the Serre functor for \(\QCoh(\punct{\Y})\) in \Cref{prop:serre_puncture}.
  This application recovers classical Serre duality for projective
  varieties.

  \ssec{Structure of the paper}

  In \Cref{sec:prelims}, we establish the notations
  and conventions used throughout the paper. We also
  recall some basic results that will be particularly
  useful. The reader is encouraged to skip this
  section and only refer back to it when necessary.

  In \Cref{sec:serre_general}, we consider formal
  properties for Serre functors of compactly
  generated DG categories, particularly
  on categories of the form \(\QCoh(\Y)\), including their
  existence.

  The main content of this note is in \Cref{sec:serre_affine}
  where we prove both \Cref{thm:serre_functor} and
  \Cref{thm:matlis_local_duality} as well as
  other properties of the Serre functor of \(\QCoh(\Y)\)
  when \(\Y \simeq \Spec A/G\) as above.

  \ssec{Acknowledgements}

  I would like to thank D. Beraldo for endless advice; many fruitful
  conversations; and for suggesting this topic.

  This work was supported by the Engineering and Physical 
  Sciences Research Council [EP/Z534882/1].

  \sec{Preliminaries}\label{sec:prelims}

  In this section, we outline the notations and conventions
  of the paper as well as defining some basic notions and
  recalling a few standard results that are worth
  highlighting. It is recommended the reader
  skip this section and only refer back to it when
  necessary.

  We always work over an algebraically closed
  field \(k\) of characteristic zero.

  \ssec{DG categories}

  As the existence of a Serre functor on compactly generated
  DG categories follows from abstract nonsense, we provide
  here some of the basics of higher category theory needed
  in order to deduce the existence of the Serre functor.
  We ignore all problems of a set-theoretic nature.

  \sssec{}

  For us, all categories are DG categories (i.e. stable \(\infty\)-categories
  with an action of \(\Vect\)) except for a few instances where 
  we work with abelian categories (which we make clear).
  We generally assume DG categories are cocomplete (i.e. 
  contain arbitrary direct sums) unless stated explicitly otherwise
  or unless they are of the form \(\Perf(\Y)\) or \(\Coh(\Y)\)
  for \(\Y\) a prestack. We assume
  functors between categories are exact and continuous (i.e. commute with
  arbitrary direct sums) unless explicitly stated otherwise or
  if they are (internal) hom functors. 

  The category of exact
  functors between \(\C\) and \(\mathcal{D}\) is denoted \(\Fun(\C, \mathcal{D})\)
  and the full subcategory of continuous functors is
  denote \(\Fun_{\text{cont}}(\C, \mathcal{D})\).

  In general, hom-spaces are denoted by \(\Hom_\C(-,-)\)
  while internal homs (when they exist) are denoted by \(\uHom_\C(-,-)\).
  When we do work in an abelian category, we use \(\sHom_\C(-,-)\) 
  and \(\usHom_\C(-,-)\) instead.

  \sssec{}
  
  Let \(\C\) be a category. An object \(c\in \C\) is \emph{compact}
  if the functor \(\Hom_\C(c,-)\) is continuous.
  The full subcategory of compact objects of \(\C\) is
  denoted by \(\cpt{\C}\).

  \sssec{}

  A family of objects \(\{c_\alpha\}\) in \(\C\) \emph{generates}
  \(\C\) if the family of functors \(\Hom_\C(c_\alpha, -)\)
  is a conservative family, that is \(\Hom_\C(c_\alpha, c) \simeq 0\)
  for all \(c_\alpha\) implies \(c \simeq 0\). This means
  that the only cocomplete full subcategory containing
  the \(c_\alpha\) is \(\C\) itself. We say \(\C\) is
  \emph{compactly generated} if there exists a family
  of compact objects generating \(\C\).

  \sssec{}

  Suppose \(\C\) is a category which is not necessarily cocomplete.
  The cocomplete category \(\Ind(\C)\) is the \emph{ind-extension}
  of \(\C\). It is equipped with a functor \(\C \to \Ind(\C)\)
  and is characterised by the property that, for all cocomplete
  categories \(\mathcal{D}\), the morphism
  \begin{displaymath}
    \Fun_{\text{cont}}(\Ind(\C), \mathcal{D}) \to \Fun(\C, \mathcal{D})
  \end{displaymath}
  is an isomorphism. The lift of a functor \(\C \to \mathcal{D}\) to 
  a continuous functor \(\Ind(\C) \to \mathcal{D}\) is
  called the \emph{ind-extension} of the functor.
  The category \(\Fun(\op{\C}, \Vect)\) is
  a model for \(\Ind(\C)\)

  \sssec{}

  The functor \(\C \to \Ind(\C)\) is fully faithful and its 
  essential image identifies with \(\cpt{\Ind(\C)}\).
  Moreover, if \(\C\) is cocomplete and compactly generated,
  then the functor \(\Ind(\cpt{\C}) \to \C\) is an
  equivalence.

  \sssec{}

  We say \(\C\) is \emph{proper} if \(x,y \in \cpt{\C}\)
  implies \(\Hom_\C(x,y) \in \cpt{\Vect}\). Recall,
  \(\cpt{\Vect}\) consists of objects with finite
  dimensional total cohomology. We say
  \(\C\) is \emph{reflexive} if \(x,y \in \cpt{\C}\)
  implies \(\Hom_\C(x,y)\) is reflexive,
  that is if each of its cohomologies
  are finite dimensional.

  \ssec{Algebraic geometry}

  \sssec{}

  Let \(\pt = \Spec k\). All of our prestacks are defined over \(k\).
  By 'stack` we mean 'algebraic stack` in the sense of \cite{drinfeld_gaitsgory:2013}.

  \sssec{}

  For a stack \(\Y\), one defines the category \(\IndCoh(\Y)\).
  It comes with a canonically defined functor \(\Psi: \IndCoh(\Y) \to \QCoh(\Y)\).
  When \(\Y\) is QCA, in the sense of \cite{drinfeld_gaitsgory:2013},
  \(\Psi\) comes from ind-extending the embedding \(\Coh(\Y) \hto \QCoh(\Y)\).

  Recall, if \(\map{f}{\Y_1}{\Y_2}\) is a morphism between stacks locally
  almost of finite type (laft) then we have a functor \(\map{f^!}{\IndCoh(\Y_2)}{\IndCoh(\Y_1)}\).
  If \(\Y\) is laft and \(p_\Y: \Y \to \pt\) is the canonical map
  then we write \(\omega_\Y\) for \(p_\Y^!(k) \in \IndCoh(\Y)\).

  The functor \(\Psi\) identifies
  \(\QCoh(\Y)\) with the left completion of \(\IndCoh(\Y)\) in
  its t-structure and induces an equivalence \(\plus{\IndCoh(\Y)} \to \plus{\QCoh(\Y)}\).
  If \(\Y\) is smooth then \(\Psi\) is an equivalence.

  Moreover, when \(\Y\) is smooth,
  \begin{displaymath}
    \Psi(\omega_\Y) \simeq \det(\bL_\Y)[\dim \Y]
  \end{displaymath}
  where \(\bL_\Y\) is cotangent complex of \(\Y\).

  \sssec{}

  If \(\Y\) is a QCA algebraic stack then the canonical
  equivalence \(\dual{\IndCoh(\Y)} \simeq \IndCoh(\Y)\)
  induces an involution \(\Dse: \op{\Coh(\Y)} \to \Coh(\Y)\).
  Explicitly, we have \(\Dse \simeq \uHom_{\QCoh(\Y)}(-, \Psi(\omega_\Y))\)
  \cite[4.4]{drinfeld_gaitsgory:2013}.

  \sssec{}

  For a prestack \(\Y\), the category \(\Perf(\Y)\) is the
  full subcategory of \(\QCoh(\Y)\) consisting of
  dualisable objects. Following \cite{ben_zvi_francis_nadler:2010},
  we say \(\Y\) is \emph{perfect} if \(\QCoh(\Y)\) is compactly
  generated and \(\Perf(\Y)\) and \(\cpt{\QCoh(\Y)}\)
  coincide.
  By \cite[Corollary 1.4.3]{drinfeld_gaitsgory:2013},
  all QCA stacks are perfect.

  \sssec{}\label{sec:quotient_notation}

  Except in \Cref{sec:serre_general}, \(\Spec A\) is a \emph{classical}
  affine scheme of finite type,
  acted on by a fixed reductive group \(G\) and \(\Y\)
  is the quotient \(\Spec A/G\).

  We remark that \(\QCoh(\Y)\) is compactly
  generated by \(\Perf(\Y)\) and \(\IndCoh(\Y)\)
  is compactly generated by \(\Coh(\Y)\).

  We have a natural map \(\pi: \Y \to BG\). As
  \(\pi_\ast\) is conservative, \(\QCoh(\Y)\)
  is compactly generated by \({\{\pi^\ast V_\lambda : \lambda \in \Lambda_G\}}\).
  Here, \(\Lambda_G\) is the set of dominant weights
  for \(G\) and, for \(\lambda \in \Lambda\), 
  \(V_\lambda\) is the corresponding irreducible
  representation. We write \(\O_\Y(\lambda)\)
  for \(\pi^\ast V_\lambda\).

  The action of \(G\) on \(\Spec A\) induces
  a \(\Gamma(G, \O_G)\)--comodule structure
  on \(A\). We write \(A^G\) for the invariants
  of this coaction. There is a map
  \(p: \Y \to \Spec A^G\) such that
  \(p_\ast\) is t-exact and preserves
  coherence \cite{alper:2013}.

  We write \(q: \Spec A \to \Y\) for the
  natural quotient/atlas map.
  Note that \(q^\ast\) is
  the forgetful functor from \(G\)--equivariant \(A\)--modules
  to \(A\)--modules.

  \ssec{Equivariant algebra}

  We will need to use some results on the category
  \(\C = \heart{\QCoh(\Y)}\), that is the abelian
  category of \(G\)--equivariant \(A\)--modules.
  These are \(G\)--equivariant version
  of many classical algebra results which
  will be useful for us to reference.

  \sssec{}

  A \(G\)--ideal of \(A\) is a \(G\)--equivariant \(A\)--module 
  contained in \(A\). We say \(A\) is \(G\)--local if it has
  a unique maximal \(G\)--ideal.

  \sssec{}

  The category \(\C\) has enough injectives (since \(\Spec A/G\) is algebraic
  with affine diagonal) and thus every object
  has an injective hull. Recall that an essential extension
  of a \(G\)--equivariant \(A\)--module \(M\) is a module \(N\)
  containing \(M\) such that every other \(G\)--submodule of 
  \(N\) intersects \(M\) non-trivially. The injective
  hull of \(M\) is then an injective essential extension.

  \sssec{}

  We say \(M \in \C\) is \(G\)-free if \(M \simeq \pi^\ast V \)
  where \(V\) is a representation of \(G\).
  Note that for every \(M \in \C\) there
  is an epimorphism from a \(G\)--free
  module to \(M\).

  Suppose \(\fm\) is a maximal \(G\)-ideal of \(A\).
  Then note that \(G\) acts transitively on \(\Spec (A/\fm)\)
  so \(\Spec (A/\fm) \simeq G/H\) for some subgroup \(H\).
  As a result, \(\Spec (A/\fm)/G \simeq BH\)
  and we can conclude that every \(G\)--equivariant
  \(A/\fm\)--module is \(G\)--free.

  \sssec{}\label{sec:equiv_nakayama}

  Let us assume that \(A\) is \(G\)--local with unique
  maximal \(G\)--ideal \(\fm\). If \(X \in \C\)
  is finitely generated and \(\fm X = X\)
  then, by Nakayama, there exists
  \(m \in \fm\) such that \(mx = x\)
  for all \(x \in X\). Note, this
  forces \(m \in A^G\) so \((\fm \cap A^G) X = X\).
  However, if \(A\) is \(G\)--local
  then \(A^G\) is local so Nakayama
  then implies \(X = 0\).

  \sssec{}

  As an application of this, suppose \(P \in \C\)
  is projective. Take \(F\) to be a \(G\)--free
  module with an epimorphism \(F \to P\).
  Let \(K\) be the kernel in \(\C\).
  In the same way one proves it with
  ordinary local rings, we see that
  \(K \subseteq \fm F\) and since \(P\)
  is projective, we deduce
  that \(\fm K = K\) in \(\C\).
  Thus, \(K = 0\). It follows
  that all projective objects in \(\C\)
  are \(G\)--free.

  \sssec{}\label{sec:socle}

  Take \(M \in \C\). The \(G\)--socle of \(M\) is taken to be
  \begin{displaymath}
    \Soc_G(M) = \usHom_\C(A/\fm, M).
  \end{displaymath}
  This is naturally a \(G\)--equivariant module
  over \(A/\fm\) and hence is \(G\)--free.
  We can also consider it as a \(G\)--submodule of \(M\).
  Hence, if \(M\) is \(G\)--Artinian, 
  \(\Soc_G(M) \simeq A/\fm \otimes_k V\)
  where \(V\) is a finite dimensional representation
  of \(G\). Note that if \(N\) is a \(G\)--submodule of \(M\),
  \(N \cap \Soc_G(M) \simeq \Soc_G(N)\) which is
  non-zero as \(N\) is also \(G\)--Artinian
  (hence it contains a simple \(G\)--submodule which must be
  \(A/\fm\)). So \(M\) is an essential extension of \(\Soc_G(M)\).
  They thus have the same injective hull. Write
  \(E\) for the injective hull of \(A/\fm\).
  Then the injective hull of \(M\) is \(E \otimes _k V\).

  \sec{Serre Functors}\label{sec:serre_general}

  In this section we cover some of the abstract theory
  of Serre functors.

  \ssec{Existence of Serre functors}

  \sssec{}

  Let \(\C\) be a compactly generated category. Then
  \begin{displaymath}
    \C \simeq \Ind(\cpt{\C}) \simeq \Fun(\op{(\cpt{\C})}, \Vect).
  \end{displaymath}
  Thus, for all \(x \in \cpt{\C}\) the functor \(\dual{\Hom_\C(x, -)}\)
  is representable by some \(\Se_\C(x)\). By ind-extension,
  the assignment \(x \mapsto \Se_\C(x)\) defines a
  continuous functor \(\map{\Se_\C}{\C}{\C}\). By construction,
  it is characterised by the property that
  \begin{displaymath}
    \dual{\Hom_\C(x,y)} \simeq \Hom_\C(y, \Se_\C(x))
  \end{displaymath}
  for all \(x \in \cpt{\C}\) and all \(y \in \C\)
  although it is enough to check this property
  with \(y \in \cpt{\C}\). When \(\C\) is obvious
  from context we just write \(\Se\) instead of \(\Se_\C\).

  \sssec{}

  Suppose \(\C\) is equipped with a symmetric monoidal
  structure with compact unit \(\mathbb{1}\). Further assume
  that all compact objects in \(\C\) are dualisable. Observe that
  \begin{displaymath}
    \dual{\Hom_\C(x,y)} \simeq \dual{\Hom_\C(\mathbb{1}, \dual{x}\otimes y)} \simeq \Hom_\C(y, \Se(\mathbb{1})\otimes x).
  \end{displaymath}
  Hence, \(\Se\) is equivalent to tensoring by \(\Se(\mathbb{1})\).

  \begin{notat}
    When \(\C\) is as above, we write \(\S_\C\) for \(\Se(\mathbb{1})\).
  \end{notat}

  \sssec{}

  In the case when \(\C\) is proper, note that, for \(x \in \cpt{\C}\),
  the functor \(\Hom_\C(x, -)\) preserves compactness and
  so has  a continuous right adjoint \(T\). See that \(T\)
  is determined by \(T(k) \in \C\) and 
  \begin{displaymath}
    \Hom_\C(y, T(k)) \simeq \Hom_\Vect(\Hom_\C(x,y), k) \simeq \dual{\Hom(x,y)}.
  \end{displaymath}
  Hence, \(T(k) = \Se(x)\). This provides another justification
  for the existence of the Serre functor when \(\C\) is proper.

  \sssec{}

  The Serre functor is fully faithful on compact objects
  if and only if \(\C\) is reflexive \cite[Proposition 1.2.6]{gaitsgory_yom_din:2017}.
  Thus, if \(\C\) is reflexive; its compact objects are dualisable
  under some symmetric monoidal structure; and the monoidal unit is compact, then
  \begin{displaymath}
    \uHom_\C(x \otimes \S_\C, y \otimes \S_\C) \simeq \uHom_\C(x,y)
  \end{displaymath}
  for all \(x,y \in \cpt{\C}\).

  \ssec{Serre Functors on QCoh}

  \sssec{}

  Let us assume \(\Y\) is a perfect stack so that the
  Serre functor for \(\QCoh(\Y)\) is given by tensoring with
  \(\S_{\QCoh(\Y)}\).

  \begin{notat}
    When \(\Y\) is as above we write \(\S_\Y\) for \(\S_{\QCoh(\Y)}\).
  \end{notat}

  \begin{prop}\label{prop:serre_under_morphism}
    Let \(\X\) and \(\Y\) be perfect stacks such that \(\map{f}{\X}{\Y}\)
    is a schematic, quasi-compact morphism. 
    Then \(f_\ast \S_\X \simeq \uHom_{\QCoh(\Y)}(f_\ast \O_\X, \S_\Y)\).
  \end{prop}
  \begin{proof}
    Suppose \(\F \in \Perf(\Y)\), Observe that
    \begin{displaymath}
      \Hom_{\QCoh(\Y)}(\F, f_\ast \S_\X) \simeq \dual{\Hom_{\QCoh(\X)}(\O_\X, f^\ast \F)}.
    \end{displaymath}
    However,
    \begin{displaymath}
      \dual{\Hom_{\QCoh(\Y)}(\O_\X, f^\ast \F)} \simeq \dual{\Hom_{\QCoh(\X)}(f^\ast \O_\Y, f^\ast \F)} \simeq \dual{\Hom_{\QCoh(\Y)}(\O_\Y, f_\ast f^\ast \F)}.
    \end{displaymath}
    Using the Serre functor on \(\QCoh(\Y)\) we deduce that 
    \begin{displaymath}
      \dual{\Hom_{\QCoh(\Y)}(\O_\X, f^\ast \F)} \simeq \Hom_{\QCoh(\Y)}(f_\ast f^\ast \F, \S_\Y).
    \end{displaymath}
    By assumptions on \(f\), we have a projection formula and thus
    \begin{displaymath}
      \dual{\Hom_{\QCoh(\Y)}(\O_\X, f^\ast \F)} \simeq \Hom_{\QCoh(\Y)}(f_\ast \O_\X \otimes \F, \S_\Y) \simeq \Hom_{\QCoh(\Y)}(\F, \uHom_{\QCoh(\Y)}(f_\ast \O_\X, \S_\Y)).
    \end{displaymath}
    Thus,
    \begin{displaymath}
      \Hom_{\QCoh(\Y)}(\F, f_\ast \S_\X) \simeq \Hom_{\QCoh(\Y)}(\F, \uHom_{\QCoh(\Y)}(f_\ast \O_\X, \S_\Y))
    \end{displaymath}
    for all \(\F\in \Perf(\Y)\), proving that 
    \(f_\ast \S_\X \simeq \uHom_{\QCoh(\Y)}(f_\ast \O_\X, \S_\Y)\).
  \end{proof}

  \sec{Serre functors on affine quotients}\label{sec:serre_affine}

  \ssec{Elementary properties of affine quotients}

  We now focus our study on the stack \(\Y\) 
  as defined in \Cref{sec:quotient_notation}. It
  is perfect so its Serre functor is
  given by tensoring with \(\S_\Y\). We would
  like to consider cases when \(\QCoh(\Y)\)
  is also proper.

  \begin{prop}
    If \(\Spec A\) has finitely many closed \(G\)--orbits then 
    \(\QCoh(\Y)\) is proper.
  \end{prop}
  \begin{proof}
    It is enough to check that \(\Hom_{\QCoh(\Y)}(\O_\Y(\lambda), \O_\Y(\mu))\)
    is compact for any \(\lambda, \mu \in \Lambda_G\). 
    By dualisability, it is enough to check
    \(\Hom_{\QCoh(\Y)}(\O_\Y, \O_\Y(\lambda))\) is compact
    for all \(\lambda \in \Lambda_G\).
    Observe that this is the vector space underlying
    \(p_\ast \O_\Y(\lambda)\), which is coherent by \cite[Theorem 4.16(x)]{alper:2013},
    and lands in the heart by t-exactness. Thus it is a finitely
    generated \(A^G\)--module. The claim is thus proven
    if \(A^G\) is finite dimensional over \(k\).

    Closed orbits of \(\Spec A\) correspond bijectively to
    points of \(\Spec A^G\) (this is \cite[Theorem 4.16(iv)]{alper:2013} 
    or \cite[Proposition 0.1, Theorem 1.1]{mumford:1994}). Thus,
    if \(\Spec A\) has finitely many closed orbits, \(\Spec A^G\)
    has finitely many maximal ideals.

    Note \(A^G\) is a finitely generated \(k\)-algebra
    (this is Hilbert's 14th problem which was proven by Nagata
    when \(G\) is linearly reductive, or it is \cite[Theorem 4.16(xi)]{alper:2013}),
    thus it has finitely many maximal ideals if and only if it is
    Artinian. Thus \(A^G\) is Artinian and Noetherian so
    is indeed finite dimensional over \(k\).
  \end{proof}

  \sssec{}

  It thus makes sense to consider when \(\Spec A\)
  has finitely many closed orbits. From now on,
  we will assume that, in particular, \(\Spec A\)
  has a unique closed orbit. Equivalently,
  \(A\) is \(G\)--local. Also, note \(A^G\) is local.

  Take a point \(x \in \Spec A\) in this orbit
  and let \(H\) be its stabiliser. By Luna's slice
  theorem, there is a locally closed affine
  subvariety \(\Spec B\) of \(\Spec A\)
  containing \(x\) such that \(\{x\}\)
  is the unique closed orbit of \(H\)
  acting on \(\Spec B\) and \(\Y \simeq \Spec B/H\).
  
  Therefore, we assume from now on that \(\Spec A\)
  has a unique closed orbit and that it is a fixed point.
  Note that this is equivalent to the unique
  \(G\)--maximal ideal of \(A\) being 
  maximal as an ordinary ideal.

  \sssec{}

  The fixed point defines a closed immersion \(i: BG \hto \Y\)
  with open complement \(j: \punct{\Y} \hto \Y\). Let \(\QCoh(\Y)_{BG}\)
  be the full subcategory of \(\QCoh(\Y)\) given by the
  kernel of \(j^\ast\). The embedding \(\QCoh(\Y)_{BG}\)
  has a right adjoint, denoted by \(\localcoh\)
  and called the local cohomology functor.
  Explicitly,
  \begin{displaymath}
    \localcoh \F \simeq \ker(\F \to j_\ast j^\ast \F).
  \end{displaymath}

  \begin{prop}[{c.f. \cite[Proposition II.4.6.1.3]{gaitsgory_rozenblyum:2017}}]\label{prop:bg_gens}
    The category \(\QCoh(\Y)_{BG}\) is generated by the essential image of \(i_\ast\).
  \end{prop}
  \begin{proof}
    The embedding \(\QCoh(\Y)_{BG} \hto \QCoh(\Y)\) is compatible
    with the t-structure since \(j^\ast\) is t-exact. Hence,
    \(\QCoh(\Y)_{BG}\) is generated by \(\heart{\QCoh(\Y)_{BG}}\).
    Any \(\F \in \heart{\QCoh(\Y)_{BG}}\) is the colimit
    of its coherent subsheaves \cite[Proposition 15.4]{laumon_moret_bailly:2005}.
    If \(\F \in \heart{\QCoh(\Y)_{BG}}\) is coherent
    then it must be annihilated by some power
    of the \(G\)--maximal ideal \(\fm\) of \(A\).
    Hence, \(\F\) has a filtration by objects
    of the form \(\fm^n \cdot \F\) and the
    quotients \(\fm^n \cdot \F / \fm^{n+1} \F\) are
    annihilated by \(\fm\) so lie in the essential
    image of \(i_\ast\).
  \end{proof}
  
  \sssec{}

  Our next small result is in a similar spirit to the one above:

  \begin{prop}\label{prop:k_generates}
    The functor \(\uHom_{\QCoh(\Y)}(i_\ast \O_{BG}, -)\) is conservative
    on \(\Coh(\Y)\).
  \end{prop}
  \begin{proof}
    Take \(\cG \in \Coh(\Y)\) and assume \(\cG \not\simeq 0\). We need to show
    \(\uHom_{\QCoh(\Y)}(i_\ast \O_{BG}, \cG)\)
    is non-zero. Notice the underlying \(A\)--module is
    \(\uHom_{A}(A/\fm, \cG)\) where \(\fm\) is the
    unique \(G\)--maximal ideal of \(A\).
    Thus it is enough to show the \(\fm\)--depth
    of the complex \(\cG\) is finite. As
    \(\cG \in \Coh(\Y)\), this follows from
    \cite[2.5]{foxby_iyengar:2003}.
    Although the argument in \emph{loc. cit.}
    assumes \((A, \fm)\) is local,
    it still follows in our case by
    \Cref{sec:equiv_nakayama}.
  \end{proof}

  One consequence of this is that
  any coherent sheaf must have sections supported
  at the fixed point of \(\Y\). More precisely:

  \begin{cor}
    The functor \(\localcoh\) is conservative on \(\Coh(\Y)\).
  \end{cor}
  \begin{proof}
    Take \(\cG \in \Coh(\Y)\) and assume \(\localcoh \cG \simeq 0\).
    Then 
    \begin{displaymath}
      \uHom_{\QCoh(\Y)}(i_\ast \O_{BG}, \localcoh \cG) \simeq \uHom_{\QCoh(\Y)}(i_\ast \O_{BG}, \cG) \simeq 0
    \end{displaymath}
    since \(\localcoh\) is right adjoint to \(\QCoh(\Y)_{BG} \hto \QCoh(\Y)\).
    However, by \Cref{prop:k_generates}, this implies
    \(\cG \simeq 0\).

    Equivalently, one can use \cite[Theorem 2.1]{foxby_iyengar:2003}
    and the argument in the proof above.
  \end{proof}

  \ssec{The Serre functor}

  We now look at the object \(\S_\Y\) itself and the functor
  \(\uHom_{\QCoh(\Y)}(-, \S_\Y)\). We introduce some notation for it:

  \begin{notat}
    Write \(\Dmat: \op{\QCoh(\Y)} \to \QCoh(\Y)\) for
    the functor \(\uHom_{\QCoh(\Y)}(-, \S_\Y)\).
  \end{notat}

    \begin{prop}\label{prop:s_inj}
    The functor \(\Dmat\) is t-exact.
    Thus \(\S_\Y \in \heart{\QCoh(\Y)}\) and is an injective object in
    this abelian category.
  \end{prop}
  \begin{proof}
    First, note that, as \(G\) is linearly reductive,
    the functor \(H_\lambda =\Hom_{\QCoh(\Y)}(\O_\Y(\lambda), -)\) is t-exact
    for all \(\lambda \in \Lambda_G\). Moreover, the
    family of functors \(\set{H_\lambda}{\lambda \in \Lambda_G}\)
    is jointly conservative. Thus \(\uHom_{\QCoh(\Y)}(-, \S_\Y)\)
    is t-exact if and only if
    \begin{displaymath}
      H_\lambda(\uHom_{\QCoh(\Y)}(-, \S_Y)) \simeq \Hom_{\QCoh(\Y)}(- \otimes \O_\Y(\lambda), \S_\Y)
    \end{displaymath}
    is t-exact for all \(\lambda \in \Lambda_G\). However, observe that
    \begin{displaymath}
      \Hom_{\QCoh(\Y)}(- \otimes \O_\Y(\lambda), \S_\Y) \simeq
      \dual{H_{\dual{\lambda}}(\F)}.
    \end{displaymath}
    Thus, as taking the linear dual is t-exact and 
    conservative, the claim follows again from the 
    t-exactness of the \(H_\lambda\).
  \end{proof}

  \begin{lem}\label{lem:supp_s}
    We have \(\S_\Y \in \QCoh(\Y)_{BG}\).
  \end{lem}
  \begin{proof}
    By \Cref{prop:s_inj}, \(\S_\Y \in \heart{\QCoh(\Y)}\)
    so we treat it as a \(G\)--equivariant \(A\)--module. 
    We will show it is \(G\)--Artinian.
    Let
    \(\fm\) be the unique maximal \(G\)--ideal of \(A\).

     Suppose we have a chain of submodules:
    \begin{displaymath}
      \cdots \hto M_3 \hto M_2  \hto M_1 \hto \S_\Y.
    \end{displaymath}
    Applying \(\Dmat\) yields epimorphisms:
    \begin{displaymath}
      \O_\Y \tto \Dmat(M_1) \tto \Dmat(M_2) \tto \Dmat(M_3) \tto \cdots.
    \end{displaymath}
    As \(A\) is Noetherian, this stabilises so there exists \(a \in \Z\) such that
    \(\Dmat(M_a) \simeq \Dmat(M_{a+1})\).
    Let \(N = M_a/M_{a+1}\). 

    Aiming for a contradiction, assume \(N \not\simeq 0\).
    Then there is a monomorphism \(A/\fm \simeq i_\ast \O_{BG} \hto N\).
    Applying \(\Dmat\) yields an epimorphism \(\Dmat(i_\ast \O_{BG}) \tto \Dmat(N) \simeq 0\).
    However,
    \begin{displaymath}
      \Dmat(i_\ast \O_{BG}) \simeq i_\ast \O_{BG}
    \end{displaymath}
    by \Cref{prop:serre_under_morphism}. Thus, this is a contradiction.

    It follows that \(M_a \simeq M_{a+1}\) so
    \(\S_\Y\) is Artinian. Thus it is, element-wise,
    annihilated by powers of \(\fm\).
  \end{proof}

  \sssec{}

  It turns out \Cref{lem:supp_s} is the last piece we need to prove
  that \(\S_\Y \simeq \localcoh \Psi(\omega_\Y)[\dim G]\) as claimed in \Cref{thm:serre_functor}.

  \begin{proof}[Proof of \Cref{thm:serre_functor}]
    As proven in \Cref{lem:supp_s}, \(\S_\Y \in \QCoh(\Y)_{BG}\)
    and we know \(\QCoh(\Y)_{BG}\) is generated by the essential image
    of \(i_\ast\) by \Cref{prop:bg_gens}.
    Thus, \(\S_\Y\) is characterised by the fact that, for all \(V \in \cpt{\Rep(G)}\),
    we have
    \begin{displaymath}
      \Hom_{\QCoh(\Y)_{BG}}(i_\ast V, \S_\Y) \simeq \dual{\Hom_{\QCoh(\Y)}(\O_\Y, i_\ast V)}
      \simeq \dual{\Hom_{\Rep(G)}(\O_{BG}, V)} \simeq \Hom_{\Rep(G)}(V, \O_{BG}).
    \end{displaymath}

    Now, \(\localcoh \Psi(\omega_\Y) \in \QCoh(\Y)_{BG}\) and
    \begin{displaymath}
      \Hom_{\QCoh(\Y)_{BG}}(i_\ast V, \localcoh \Psi(\omega_\Y)) \simeq \Hom_{\QCoh(\Y)}(i_\ast V, \Psi(\omega_\Y))
    \end{displaymath}
    since \(\localcoh\) is the right adjoint of \(\QCoh(\Y)_{BG} \hto \QCoh(\Y)\).
    Then, by \((i^{\IndCoh}_\ast, i^!)\)--adjunction,
    \begin{displaymath}
      \Hom_{\QCoh(\Y)}(i_\ast V, \Psi(\omega_\Y)) \simeq \Hom_{\IndCoh(\Y)}(i^{\IndCoh}_\ast V, \omega_\Y) \simeq \Hom_{\Rep(G)}(V, \Psi(\omega_{BG})).
    \end{displaymath}
    However, \(\Psi(\omega_{BG}) \simeq \O_{BG}[-\dim G]\) and, therefore,
    \begin{displaymath}
      \Hom_{\QCoh(\Y)_{BG}}(i_\ast V, \localcoh \Psi(\omega_\Y)) \simeq \Hom_{\Rep(G)}(V, \O_{BG}[-\dim G]).
    \end{displaymath}

    Hence, \(\S_\Y \simeq \localcoh \Psi(\omega_\Y)[\dim G]\).
  \end{proof}

  \sssec{}

  Our next minor result is how \(\S_\Y\) is the
  injective hull of \(i_\ast \O_{BG}\) in the abelian
  category \(\heart{\QCoh(\Y)}\) which explains
  the connection to classical Matlis duality.

  \begin{prop}
    In the abelian category \(\heart{\QCoh(\Y)}\), \(\S_\Y\)
    is indecomposable, that is it cannot be written as a non-trivial
    direct sum.
  \end{prop}
  \begin{proof}
    Since \(\usHom(\S_\Y, \S_\Y) \simeq \O_\Y\), if \(\S_\Y\)
    were decomposable then \(\O_\Y\) would be too. If \(\O_\Y\)
    were decomposable, \(A\) would contain a non-trivial,
    homogeneous idempotent
    which must be in \(A^G\). However, \(A^G\) is local and
    so cannot contain a non-trivial idempotent.
  \end{proof}

  \begin{cor}\label{cor:hull}
    In the abelian category \(\heart{\QCoh(\Y)}\), 
    \(\S_\Y\) is the injective hull of \(i_\ast \O_{BG}\).
  \end{cor}
  \begin{proof}
    Since we identify \(i_\ast \O_{BG}\) with \(A/\fm\)
    where \(\fm\) is the unique, maximal \(G\)--ideal
    of \(A\), we see that the map \(\O_\Y \to i_\ast \O_{BG}\)
    is an epimorphism in \(\heart{\QCoh(\Y)}\).
    Thus \(\usHom(i_\ast \O_{BG}, \S_\Y) \to \S_\Y\)
    is a monomorphism.

    As \(\S_\Y\) is injective, \(\usHom(i_\ast \O_{BG}, \S_\Y)
    \simeq \uHom(i_\ast \O_{BG}, \S_\Y)\) which is
    \(i_\ast \O_{BG}\) by \Cref{prop:serre_under_morphism}.
    Thus we have a monomorphism \(i_\ast \O_{BG} \hto \S_\Y\).
    Since \(\S_\Y\) is injective, this means \(i_\ast \O_{BG}\)
    is a direct summand of \(\S_\Y\) but, by the result above,
    \(\S_\Y\) is indecomposable so \(\S_\Y\) is an essential
    extension of \(i_\ast \O_{BG}\). Therefore, it
    is the injective hull.
  \end{proof}

  \ssec{Local and Matlis duality}

  Using the description of \(\S_\Y\) given
  by \Cref{thm:serre_functor},
  we will now go on to prove \Cref{thm:matlis_local_duality}
  which relates \(\Dmat\) with classical Serre duality.
  First, let us recall some notation:

  \begin{notat}
    We write \(\Dse: \op{\QCoh(\Y)} \to \QCoh(\Y)\) for the functor
    \(\uHom_{\QCoh(\Y)}(-, \Psi(\omega_\Y))\).
  \end{notat}

  The main obstacle in proving \Cref{thm:matlis_local_duality} 
  is showing that, if \(\F \in \Coh(\Y)\),
  then \(\Dmat(\localcoh \F) \in \Coh(\Y)\).
  For this, we need to prove that \(\uHom_{\QCoh(\Y)}(\F, \Dmat(\localcoh \O_\Y))\)
  lies in \(\Coh(\Y)\) which follows if we show
  the cohomology of \(\Dmat(\localcoh \O_\Y)\) is coherent
  and, also, that \(\uHom_{\QCoh(\Y)}(-, \Dmat(\localcoh \O_\Y))\)
  is bounded above when restricted to \(\heart{\Coh(\Y)}\)
  (see, for instance, \cite[Proposition II.7.20]{hartshorne:1966} for details).
  Before we do this, we need the following lemma:

  \begin{lem}\label{lem:underlying_loc}
    Let \(q: \Spec A \to \Y\) be the natural map. For \(\F \in \QCoh(\Y)\),
    we have \(q^\ast \localcoh \F \simeq \Gamma_{\fm} q^\ast \F \).
  \end{lem}
  \begin{proof}
    By the definition of \(\localcoh\), 
    \begin{displaymath}
      q^\ast \localcoh \F \simeq \ker(q^\ast \F \to q^\ast j_\ast j^\ast \F).
    \end{displaymath}
    Let \(s: U \hto \Spec A\) be the open complement of the
    fixed point of the \(G\)--action and let \(r: U \to \punct{\Y}\)
    be the quotient map. Then note that
    \begin{displaymath}
      \begin{tikzcd}
        U \arrow[r, "r"] \arrow[d, "s"] & \punct{\Y} \arrow[d, "j"] \\
        \Spec A \arrow[r, "q"]                          & \Y
      \end{tikzcd}
    \end{displaymath}
    is Cartesian so \(q^\ast j_\ast \simeq s_\ast r^\ast\) and thus
    \(q^\ast j_\ast j^\ast \F \simeq s_\ast r^\ast j^\ast \F \simeq s_\ast s^\ast q^\ast \F\).
    
    Now,
    \begin{displaymath}
      \Gamma_{\fm} q^\ast \F \simeq \ker(q^\ast \F \to s_\ast s^\ast q^\ast \F)
    \end{displaymath}
    and, since all maps are from adjunction, this proves the claim.
  \end{proof}

  \begin{lem}\label{lem:dloc_o_fg}
    The cohomology of \(\Dmat(\localcoh \O_\Y)\) is finitely generated.
  \end{lem}
  \begin{proof}
    Observe that, since \(\S_\Y \in \heart{\QCoh(\Y)}\),
    we have a spectral sequence
    \begin{displaymath}
      E_2^{a,b} = H^a(\uHom_{\QCoh(\Y)}(H^{-b}(\localcoh\O_\Y), \S_\Y))
      \implies H^{a+b}(\uHom_{\QCoh(\Y)}(\localcoh\O_\Y, \S_\Y)).
    \end{displaymath}
    By \Cref{prop:s_inj},
    \(\uHom_{\QCoh(\Y)}(H^{-b}(\localcoh \O_\Y), \S_\Y)
    \simeq \usHom_{\QCoh(\Y)}(H^{-b}(\localcoh \O_\Y), \S_\Y))\),
    so it is enough to show \(\usHom_{\QCoh(\Y)}(H^{b}(\localcoh\O_\Y), \S_\Y))\)
    is finitely generated for all \(b\).

    Assume for now that \(H^b(\localcoh\O_\Y)\) is \(G\)--Artinian
    as an object of \(\heart{\QCoh(\Y)}\). Then, by \Cref{sec:socle}
    and \Cref{cor:hull}, we have a monomorphism
    \(H^b(\localcoh \O_\Y) \hto \S_\Y \otimes \pi^\ast V\)
    for some \(V \in \Rep(G)\) finite dimensional. Applying
    \(\Dmat\) yields an epimorphism
    \begin{displaymath}
      \pi^\ast \dual{V} \to \usHom_{\QCoh(\Y)}(H^b(\localcoh \O_\Y),\S_\Y)
    \end{displaymath}
    and hence it is finitely generated.

    Therefore, we just need to show \(H^b(\localcoh \O_\Y)\)
    is \(G\)--Artinian. As \(q^\ast\) is t-exact, from \Cref{lem:underlying_loc},
    \begin{displaymath}
      q^\ast H^b(\localcoh \O_\Y) \simeq H^b(\Gamma_\fm A)
    \end{displaymath}
    and this is an Artinian \(A\)--module (localise and 
    use \cite[Proposition 3.5.4]{bruns_herzog:1998}) which proves the claim.
  \end{proof}

  \begin{lem}\label{lem:fid}
    There exists \(n \in \Z\) such that, for all \(\F \in \heart{\QCoh(\Y)}\),
    one has \(\Dmat(\localcoh \F) \in \QCoh(\Y)^{[-n,n]}\).
  \end{lem}
  \begin{proof}
    By \Cref{prop:s_inj},
    all we need is the cohomology
    of \(\localcoh \F\) to be bounded.
    It is enough to check \(q^\ast \localcoh \F\)
    has bounded cohomology as \(q^\ast\)
    is t-exact and conservative.
    From \Cref{lem:underlying_loc},
    \(q^\ast \localcoh \F \simeq \ker(q^\ast \F \to s_\ast s^\ast q^\ast \F)\),
    using the notation of \Cref{lem:underlying_loc}.
    Note then we just want to show \(s_\ast s^\ast q^\ast \F\)
    has bounded cohomology.
    However, \(\Gamma(\Spec A, -)\) is exact and conservative
    so this is equivalent to checking 
    \begin{displaymath}
      \Gamma(\Spec A, s_\ast r^\ast j^\ast \F) \simeq \Gamma(U, r^\ast j^\ast \F)
    \end{displaymath}
    is cohomologically bounded but this clearly lies in \(\Vect^{[0, \dim U]}\).
  \end{proof}

  \begin{prop}\label{prop:dloc_coh}
    If \(\F \in \Coh(\Y)\), then \(\Dmat(\localcoh \F) \in \Coh(\Y)\).
  \end{prop}
  \begin{proof}
    Since \(\localcoh \F \simeq \localcoh \O_\Y \otimes \F\), we see that
    \begin{displaymath}
      \Dmat(\localcoh \F) \simeq \uHom_{\QCoh(\Y)}(\F, \Dmat(\localcoh \O_\Y)).
    \end{displaymath}
    Thus, all we need to prove is that
    \begin{itemize}
      \item the cohomology of \(\Dmat(\localcoh \O_\Y)\) is coherent;
      \item and the functor \(\uHom_{\QCoh(\Y)}(-, \Dmat(\localcoh \O_\Y))\)
        is bounded above
    \end{itemize}
    (see \cite[Proposition II.7.20]{hartshorne:1966} for details).
    However, we proved these in \Cref{lem:dloc_o_fg} and
    \Cref{lem:fid}.
  \end{proof}

  \sssec{}

  We can now prove the first part of \Cref{thm:matlis_local_duality}:

  \begin{prop}\label{prop:dmat_coh_se}
    For all \(\F \in \Coh(\Y)\), there is a natural equivalence \(\Dmat(\localcoh(\F)) \simeq \Dse(\F)[\dim G]\).
  \end{prop}
  \begin{proof}
    Observe that 
    \begin{displaymath}
      \uHom_{\QCoh(\Y)}(\localcoh \F, \S_\Y) \simeq \uHom_{\QCoh(\Y)}(\F, \uHom_{\QCoh(\Y)}(\localcoh \O_\Y, \S_\Y))
    \end{displaymath}
    since \(\localcoh \F \simeq \localcoh \O_\Y \otimes \F\). Therefore, it
    is enough to show \(\Dmat(\localcoh \O_\Y) \simeq \Psi(\omega_\Y)[\dim G]\).
    From \Cref{thm:serre_functor}, we know \(\S_\Y \simeq \localcoh \Psi(\omega_\Y)[\dim G]\)
    and therefore, as \(\localcoh \O_\Y \in \QCoh(\Y)_{BG}\),
    \begin{displaymath}
      \uHom_{\QCoh(\Y)}(\localcoh \O_\Y, \S_\Y) \simeq \uHom_{\QCoh(\Y)}(\localcoh \O_\Y, \Psi(\omega_\Y)[\dim G]) \simeq \Dse(\localcoh \O_\Y)[\dim G].
    \end{displaymath}

    Now, note
    \begin{displaymath}
      \Dse(\localcoh \O_\Y) \simeq \cone(\Dse(j_\ast \O_{\punct{\Y}}) \to \Psi(\omega_\Y))
    \end{displaymath}
    so the claim is proven if we show \(\Dse (j_\ast \O_{\punct{\Y}})\simeq 0\).
    However, \(\Psi(\omega_\Y) \in \Coh(\Y)\) and \(\Dse(\localcoh \O_\Y) \simeq \Dmat(\localcoh \O_\Y)[-\dim G] \in \Coh(\Y)\)
    by \Cref{prop:dloc_coh} so we also know 
    \begin{displaymath}
      \Dse(j_\ast \O_{\punct{\Y}}) \in \Coh(\Y).
    \end{displaymath}
    We are then done since
    \begin{displaymath}
      \uHom_{\QCoh(\Y)}(i_\ast \O_{BG}, \Dse(j_\ast \O_{\punct{\Y}})) \simeq \uHom_{\QCoh(\Y)}(j_\ast j^\ast i_\ast \O_{BG}, \Psi(\omega_\Y)) \simeq 0
    \end{displaymath}
    by base change and this implies \(\Dse(j_\ast \O_{\punct{\Y}}) \simeq 0\) by \Cref{prop:k_generates}
  \end{proof}

  \sssec{}

  To prove the other part of \Cref{thm:matlis_local_duality},
  we will need a straightforward but
  technical lemma which we record here:

  \begin{lem}\label{lem:way_out}
    Let \(F,G: \Coh(\Y) \to \QCoh(\Y)\) be bounded above functors and \(\eta: F \to G\)
    an arrow between them. Then \(\eta\) is an isomorphism
    if \(F(\O_\Y(\lambda)) \xto{\eta} G(\O_\Y(\lambda))\) is
    for all \(\lambda \in \Lambda_G\).
  \end{lem}
  \begin{proof}
    See \cite[I.7.1]{hartshorne:1966}. The argument is
    simply an induction on amplitude. The only thing we
    add is the observation that every \(\F \in \heart{\Coh(\Y)}\)
    is the target of an epimorphism from an object of the
    form \(\pi^\ast V\) with \(V\) a finite dimensional
    representation of \(G\), which can be written
    as a finite direct sum of sheaves of the form \(\O_\Y(\lambda)\).
  \end{proof}

  With this, we can complete the proof:

  \begin{proof}[Proof of \Cref{thm:matlis_local_duality}]
    Take \(\F \in \Coh(\Y)\),
    we know that \(\Dmat(\localcoh \F) \simeq \Dse(\F)[\dim G]\)
    from \Cref{prop:dmat_coh_se}.
    We will prove that \(\Dmat(\F) \simeq \localcoh(\Dse(\F))[\dim G]\)
    from which the rest of the statement will follow since \(\Dse\)
    is an involution on \(\Coh(\Y)\).

    Observe that, as \(\S_\Y \simeq \localcoh \Psi(\omega_\Y)[\dim G]\), we have
    \begin{displaymath}
      \Dmat(\F) \simeq \ker(\Dse(\F) \to \uHom_{\QCoh(\Y)}(\F, j_\ast j^\ast \Psi(\omega_\Y)))[\dim G].
    \end{displaymath}
    By the projection formula, applying \(j_\ast j^\ast\)
    to the canonical map \(\Dse(\F) \otimes \F \to \Psi(\omega_\Y)\)
    yields a morphism
    \begin{displaymath}
      j_\ast j^\ast \Dse(\F) \to \uHom_{\QCoh(\Y)}(\F, j_\ast j^\ast \Psi(\omega_\Y)).
    \end{displaymath}
    We will show this is an isomorphism using \Cref{lem:way_out},
    which requires that both the functors \(j_\ast j^\ast \Dse\) and
    \(\uHom_{\QCoh(\Y)}(-, j_\ast j^\ast \Psi(\omega_\Y))\)
    are bounded above.

    First, see that
    \begin{displaymath}
      \uHom_{\QCoh(\Y)}(-, j_\ast j^\ast \Psi(\omega_\Y)) \simeq \cone(\Dmat[-\dim G] \to \Dse).
    \end{displaymath}
    Both \(\Dmat\) and \(\Dse\) are bounded above by \Cref{prop:s_inj} and \Cref{lem:fid}
    so we conclude that the functor \(\uHom_{\QCoh(\Y)}(-, j_\ast j^\ast \Psi(\omega_\Y))\) is bounded above.
    For \(j_\ast j^\ast \Dse\),
    using the same arguments of \Cref{lem:fid} shows that,
    as \(\Dse\) is bounded above, so is \(j_\ast j^\ast \Dse\).

    Therefore, for \(j_\ast j^\ast \Dse \to \uHom_{\QCoh(\Y)}(-, j_\ast j^\ast \Psi(\omega_\Y))\)
    to be an isomorphism, we need to only check it on \(\O_\Y(\lambda)\)
    where it is clear.

    Thus, the map
    \begin{displaymath}
      \Dse \to \uHom_{\QCoh(\Y)}(-, j_\ast j^\ast \Psi(\omega_\Y))
    \end{displaymath}
    factors through a map
    \begin{displaymath}
      \Dse \to j_\ast j^\ast \Dse
    \end{displaymath}
    and it is clear this is the map from adjunction.
    So we see that
    \begin{displaymath}
      \Dmat(\F) \simeq \ker(\Dse(\F) \to j_\ast j^\ast \Dse(\F))[\dim G] \simeq \localcoh(\Dse(\F))[\dim G]
    \end{displaymath}
    as claimed.
  \end{proof}
  
  \ssec{Serre functor on the puncture}

  To end, we use \Cref{thm:serre_functor} and \Cref{thm:matlis_local_duality}
  to find the Serre functor on \(\QCoh(\punct{\Y})\).

  \begin{prop}\label{prop:serre_puncture}
    The Serre functor on \(\QCoh(\punct{\Y})\) is given by tensoring with \(\Psi(\omega_{\punct{\Y}})[\dim G - 1]\).
  \end{prop}
  \begin{proof}
    By \Cref{prop:serre_under_morphism}, \(j_\ast \S_{\punct{\Y}} \simeq \Dmat(j_\ast \O_{\punct{\Y}})\).
    Hence, \(\S_{\punct{\Y}} \simeq j^\ast \Dmat(j_\ast \O_{\punct{\Y}})\). Notice that
    \begin{displaymath}
      j^\ast \Dmat(j_\ast \O_{\punct{\Y}}) \simeq j^\ast\ker(\Dmat(\O_\Y) \to \Dmat(\localcoh \O_{BG}))
      \simeq \ker(j^\ast \S_\Y \to j^\ast \Dmat(\localcoh \O_{BG})).
    \end{displaymath}
    However, since \(\S_\Y \simeq \localcoh \Psi(\omega_\Y)\) by \Cref{thm:serre_functor},
    we have \(j^\ast \S_\Y \simeq 0\) and thus \(\S_{\punct{\Y}} \simeq j^\ast\Dmat(\localcoh \O_{BG})[-1]\).
    From \Cref{thm:matlis_local_duality}, \(j^\ast\Dmat(\localcoh \O_{BG})[-1] \simeq \Psi(\omega_{\punct{\Y}})[\dim G -1]\).
  \end{proof}

  \begin{rmk}
    Note that this recovers classical Serre duality for projective varieties by taking \(\Spec A\)
    to be the affine cone and \(G = \G_m\).
  \end{rmk}

  \bibliographystyle{unsrt}
  \bibliography{refs}

\end{document}